\documentclass[12pt,a4paper]{article}
\usepackage[OT1]{fontenc}
\usepackage[english]{babel}
\usepackage{amsmath}
\usepackage{amsfonts}
\usepackage{amsthm}

\usepackage[bookmarks,pdfpagelabels,linkcolor={rgb}{0.1,0.6,0.86},citecolor={rgb}{0,0.33,0.55}]{hyperref}

\newtheorem{teo}{Theorem}[section]
\newtheorem*{teo*}{Theorem}
\newtheorem{lem}[teo]{Lemma}
\newtheorem{cor}[teo]{Corollary}
\newtheorem{pro}[teo]{Proposition}

\theoremstyle{definition}
\newtheorem{fed}[teo]{Definition}
\theoremstyle{remark}
\newtheorem{rem}[teo]{Remark}

\newtheorem{exa}[teo]{Example}

\def\eps{\varepsilon}

\def\R{\mathbb{R}}
\def\C{\mathbb{C}}

\def\cH{\mathcal{H}}
\def\cI{\mathcal{I}}
\def\cJ{\mathcal{J}}
\def\cK{\mathcal{K}}

\def\cL{\mathcal{L}}

\def\cS{\mathcal{S}}

\def\cM{\mathcal{M}}

\def\cU{\mathcal{U}}

\def\ese{\mathcal{S}}

\def\noi{\noindent}

\def\bdem{\begin{proof}}
\def\edem{\renewcommand{\qed}{\hfill $\blacksquare$}
\end{proof}}

\newcommand{\peso}[1]{ \quad \text{ #1 } \quad }
\newcommand{\sub}[2]{{#1}_{\mbox{\tiny{${#2}$}}}}

\DeclareMathOperator{\sgn}{sgn}
\DeclareMathOperator{\ran}{ran}
\DeclareMathOperator{\rank}{rank}

\newcommand{\pint}[1]{\displaystyle \left \langle #1 \right\rangle}
\newcommand{\PI}[2]{\left\langle #1 , #2 \right\rangle}

\DeclareMathOperator*{\smayo}{\prec}

\title{Approximation by partial isometries and symmetric approximation of finite frames}
\author{Jorge Antezana, Eduardo Chiumiento}
\date{}

\begin{document}

\maketitle 

\begin{abstract}
We solve the problem of best approximation by partial isometries of given rank to an arbitrary rectangular matrix, when the distance is measured in any unitarily invariant norm. In the case where the norm is strictly convex, we parametrize all the solutions. In particular, this allow us to  give a simple necessary and sufficient condition for uniqueness. We then apply these results to solve the global problem of approximation by partial isometries, and to  extend  the notion of symmetric approximation of frames  introduced in  M. Frank, V. Paulsen, T. Tiballi, {\it Symmetric Approximation of frames and bases in Hilbert Spaces}, Trans. Amer. Math. Soc. 354 (2002), 777-793. In addition, we characterize symmetric approximations of frames belonging to a prescribed subspace.
\end{abstract}



\section{Introduction}

%
%
%
%


In 1947, L\"owdin \cite{L47} or \cite{L70}, introduced a method to obtain an orthonormal basis from a given basis  motivated by problems arising in quantum chemistry. Nowadays this method is known as \textit{L\"owdin orthogonalization}, and it has two relevant advantages with respect to the classical Gram-Schmidt process: first, it is an order-independent procedure, and second, the constructed orthonormal basis is characterized as the closest orthonormal basis to the given basis.  Frank,  Paulsen and Tiballi \cite{FPT02} generalized L\"odwin orthogonalization to frames in Hilbert spaces, under the name of \textit{symmetric approximation of frames}. As we will show below, this might be seen as a best approximation problem by partial isometries. The aim of this paper is to study best approximation problems by partial isometries, and apply the results to  symmetric approximation of frames.


A family of vectors $\{ f_j\}_1^n$ spanning a subspace $\cK\subseteq \C^m$ is called a \textit{finite frame for} $\cK$. In general, frames are redundant spanning sets, and this is the key property which makes them useful to applications such as internet coding, quantum computing, filter banks, robust transmission and speech recognition
(see e.g. \cite{CK12, K1, K2}).  If the frame $\{ f_j\}_1^n$  satisfies the Parseval identity 
$$
\|f\|^2=\sum_{j=1}^n|\pint{f,f_j}|^2
$$
for every $f\in\cK$, then it is called a \textit{Parseval frame} (\textit{or normalized tight frame}). For this class of frames, it is not difficult to prove that for every $f\in\cK$, 
$$
f=\sum_{j=1}^n\pint{f,f_j}f_j.
$$
This may be interpreted as a reconstruction formula, which shows that $f$ can be recovered from the coordinates $\{\pint{f,f_j}\}_{1}^n$ with the same formula used for orthonormal bases. It is important to remark that general frames still have a reconstruction formula comparable to that of bases; however its expression turns out to be more complicated. This is one of the main reasons why Parseval frames are relevant in applications.




The method of symmetric approximation of frames consists in finding the closest Parseval frame to a given frame.   
More precisely,  a Parseval frame $\{ u_j \}_1^n$ is said to be a symmetric approximation of a frame $\{ f_j \}_1^n$ for a subspace $\cK \subseteq \C^m$
if 
\begin{equation*}
\sum_{j=1}^n \| f_j - u_j\|^2 \leq \sum_{j=1}^n \| f_j - x_j\|^2
\end{equation*}
for all the Parseval frames $\{ x_j\}_1^n$ which are \textit{weakly similar} to $\{ f_j\}_1^n$. This latter condition means that there is an invertible matrix $T$ satisfying $Tf_j=x_j$, $j=1, \ldots, n$, where $\cL$ is the spanning subspace of $\{x_j\}_1^n$. We would like to emphasize that, a priori, the subspace spanned by the Parseval frame $\{ u_j \}_1^n$ does not need to be the subspace $\cK$.
Note that, as this method is order-independent,  it becomes useful for many types of frames, such as those associated with finite groups
\cite{F}, where there is no canonical  order of vectors to apply Gram-Schmidt.

 The symmetric approximation of frames can be rewritten as a best approximation problem by partial isometries using the Frobenious norm  $\| \, \cdot \, \|_2$. Recall that an $m \times n$ matrix $X$ is a \textit{partial isometry} if $\|Xf\|=\|f\|$ for all $f \in \ker(F)^\perp$.
Let $F$ be the synthesis matrix of the frame $\{ f_j\}_1^n$ for a subspace $\cK \subseteq \C^m$,  i.e. $F$ is the $m \times n$ matrix whose columns are the vectors  $f_1, \ldots, f_n$.  It is straightforward to prove that the synthesis matrix of a Parseval frame for some subspace is a partial isometry. Then, the Parseval frame $\{  u_j\}_1^n$ with synthesis matrix $U$ is a symmetric approximation of $\{ f_j \}_1^n$ if
$$
\|F - U\|_2 \leq \|F-X\|_2
$$ 
for all $m \times n$ partial isometries $X$ such that $\ker(X)=\ker(F)$.  This condition on the kernels of the synthesis matrices is easily seen equivalent to being the respective frames weakly similar.    The formula for the unique  symmetric approximation $\{ u_j\}_1^n$ proved in \cite{FPT02} is given by $u_j=Ue_j$, $j=1,\ldots, n$, where $\{e_j \}_1^n$ is the standard basis of $\C^n$ and  $F=U|F|$ is the \textit{canonical polar decomposition} (i.e. $U$ is the unique partial isometry satisfying $F=U|F|$ and $\ker(U)=\ker(F)$). This frame $\{ u_j\}_1^n$ is called the \textit{canonical Parseval frame} associated to $\{ f_j\}_1^n$, and satisfies that span $\{ u_j \}_1^n$ = span $\{ f_j\}_1^n$. 
 Now we remark that L\"odwin orthogonalization may be thought as a special case when $\{ f_j \}_1^n$ is  a basis, and  $\{ u_j\}_1^n$ turns out to be an orthonormal basis (see \cite{GL}).

\subsection{Main results of this paper}

The authors in \cite{FPT02} showed a simple example to illustrate that the canonical Parseval frame is not the closest Parseval frame if one allows to consider  non weakly similar frames in the minimization problem.  This fact leads to the following question: which is the closest Parseval frame to a given frame when one  considers non weakly similar Parseval frames? On the other hand, the assumption of weakly similarity  conserves the redundancy of the given frame $\{ f_j \}_1^n$, i.e. the linear dependence of the vectors $f_1, \ldots, f_n$. But the redundancy, and thus the dimension of the subspace spanned by $\{ f_j \}_1^n$, are very sensitive to perturbations. 
It is therefore  natural  to ask  which is the closest Parseval frame to $\{ f_j\}_1^n$ with a fixed redundancy.

We address these questions in terms of  best approximation problems by partial isometries. Further, we consider more general norms than the Frobeniuous norm, the class of unitarily invariant norms (see Section \ref{Prel and not} for the definition). For $k=1, \ldots, q=\min\{ m, \,n \}$, denote by  $\cI_{m,n}^k$  the set of $m \times n$ partial isometries of rank $k$.  
The results of this paper are the following. 
\begin{itemize}
\item A solution to the best approximation problem by partial isometries  from $\cI_{m,n}^k$ for every unitarily invariant norm (Theorem \ref{teo 1 matrix});
\item A parametrization of all the minimizers in the previous problem when the norm is defined by a strictly convex symmetric gauge function, and in particular, a uniqueness condition (Theorems \ref{teo 1 sub 1 matrix} and \ref{parametrization of min});
\item A solution to the global best approximation problem by  partial isometries for every  unitarily invariant norm (Theorem \ref{teo 3 matrix});
\item An extension of the method of symmetric approximation of frames to the following families: Parseval frames with fixed redundancy,  all the Parseval frames and Parseval frames in a fixed subspace (Corollaries \ref{fixed ex} and \ref{frames global}, Theorem \ref{fix the subspace}). 
\end{itemize}  

The main theoretical tools for the proofs are the singular value decomposition and the Lidskii-Mirsky-Wielandt theorem \cite{B97, LM99}. Uniqueness results rely on the analysis of the equality case in this latter theorem recently carried out in \cite{MRS14}. The solution of the lower rank and global best approximation problems by partial isometries may be described using the canonical partial isometry associated to a best lower rank approximation; meanwhile in the higher rank approximation case the solutions are given by the partial isometries associated to (non canonical) polar decompositions (see Remark \ref{canonical of best rank app}).


  
\subsection{Previous related results}

	Problems of best approximation by partial isometries  have been considered by several authors. 
In the case where $k=n\leq m$,  it is well known that the isometric factor of a polar decomposition of $F$ solves the best approximation problem by partial isometries from $\cI_{m,n}^k$.  
Here, by a polar decomposition we mean a (in general non unique) representation $F=U|F|$, where $U$ is an isometry (i.e. $U^*U=I$) and $|F|=(F^*F)^{1/2}$. The earliest  result of this type is the Fan-Hoffmann theorem  \cite{FH55}, which established  the case $m=n$. Later on, Rao \cite{R80} stated without proof the case $n<m$. This  was proved more recently by   Laszkiewicz and Zi\c{e}tak  \cite{LZ06}. They also treated the case $k=\rank(F)$, and showed that  the partially isometric factor of the canonical polar decomposition is a solution. For applications in matrix ODEs which have orthogonal solutions we refer to \cite[Section 2.6]{H08} and the references therein.    
A special class of solutions to the global best approximation by the partial isometries were also given in \cite{LZ06}.
On the other hand, little attention has been paid in the literature to uniqueness questions, or parametrization of the minimizers in the case of multiple minimizers. For the fixed rank problem, it is known that when $F$ has full rank and $k=n\leq m$,  there is an unique minimizer  (see  \cite[Thm 8.2]{H08}).   

Other results concerning best approximation by partial isometries were proved for  operators  defined in infinite dimensional spaces and special classes of norms (see \cite{M89, W86}). The notion of symmetric approximation of frames in infinite dimensional Hilbert spaces was also developed in \cite{FPT02}. In addition,  best approximation of frames with additional structure was studied  by Janssen and Strohmer \cite{JS02}, and by Han \cite{H1, H2}.
Although the results obtained in our paper are in the finite dimensional setting, they can be also applied to infinite dimensional settings where, due to some invariance, the problem can be reduced to finite dimensional problems. Some examples of this situation, of major relevance in harmonic analysis, are Gabor frames, shift-invariant frames and wavelets by Parseval frames of the same class (see \cite{Bo}, \cite{CP} and \cite{RS}).


\bigskip

The paper is organized  as follows. In Section \ref{Prel and not} we establish notation and give the necessary background. 
In Section \ref{main} we present the  results concerning approximation by partial isometries. Section \ref{sec frames} is devoted to the applications in symmetric approximation of frames. In Section \ref{proofs} we prove the main results.






\section{Preliminaries and notation}\label{Prel and not}

Let $\cM_{m,n}$ be the space of  complex $m \times n$ matrices. When $m=n$, we  write $\cM_n$.
Set $q=\min\{ m,\, n \}$. 
Given  $F \in \cM_{m,n}$, 
the vector of  singular values of $F$ (i.e. the  eigenvalues of $|F|=(F^*F)^{1/2}$) arranged in nonincreasing order is given by $s(F)=(s_1(F), \ldots, s_q(F))$. Note that there are at most $q$ non-zero eigenvalues of $|F|$, so we are only  taken into account its largest $q$ eigenvalues.

\medskip

\noi \textbf{Partial isometries}. A matrix $X \in \cM_{m,n}$ is a \textit{partial isometry} if $\| Xf\|=\|f\|$ for all $f \in \ker(X)^{\perp}$. 
This is equivalent to saying that $X^*X$ is a projection (or $XX^*$ is a projection).
Another characterization of a partial isometry is as a matrix whose singular values are all $0$ or $1$.  We will use the following notation
$$
\cI_{m,n}=\{ \, X \in \cM_{m,n}   \,  : \, X \text{ is a partial isometry}  \,   \}.
$$
The connected components of the set of all partial isometries 
are determined by the rank. We denote each connected component by
$$
\cI_{m,n}^k=\{ \, X \in \cI_{m,n} \,  : \, \rank(X)=k  \,  \},
$$
where $k=1, \ldots , q$.  Let $\cU_n$ denote the group of $n \times n$ unitary matrices.  We note  that there is a left action of the group $\cU_m \times \cU_n$   on $\cI_{m,n}$ given by 
$$
(V,W) \cdot X:=VXW^* , \, \, \, \, V \in \cU_m, W \in \cU_m , \, X \in \cI_{m,n}\, .
$$
Furthermore, the orbits of this action coincide with the connected components of $\cI_{m,n}$ (see for instance \cite{HM63}).

\medskip

\noi \textbf{Singular value decomposition.}
 Let $F \in \cM_{m,n}$ be of rank $r$. Then $F$ has a \textit{singular value decomposition (SVD)}   
$$ F=V   \Sigma    W^*,$$ 
where 
$$
\Sigma=\begin{bmatrix}  \Sigma_r & 0 \\ 0 & 0_{m-r,n-r}  \end{bmatrix}, 
$$
$V \in \cU_m$, $W \in \cU_n$ and $\Sigma_r=diag(s_1(F), \ldots, s_r(F))$. The columns $\{w_1,\ldots,w_n\}$ of $W$ form an orthonormal basis of eigenvectors of $F^*F$, while the columns $\{v_1,\ldots,v_m\}$ of $V$ are eigenvectors of $FF^*$.

\begin{rem}\label{unitaries SVD}
Note that, although the matrix $\Sigma$ is completely determined by the non-zero singular values, the unitary matrices are not unique.
If $F=\tilde{V} \Sigma \tilde{W}^*$ is another SVD, then
$$
\tilde{V}=V\begin{bmatrix} D & 0 \\ 0 & R_1  \end{bmatrix}, \, \, \, \, \, \, \tilde{W}=W \begin{bmatrix} D & 0 \\ 0 & R_2  \end{bmatrix},
$$ 
where $D \in \cU_r$ is  block diagonal  with $D_{ij}=0$ if $s_i(F)\neq s_j(F)$, and $R_1 \in \cU_{m-r}$, $R_2 \in \cU_{n-r}$ are arbitrary.
\end{rem}

\medskip

\noi \textbf{The polar decomposition.} Given a matrix $F \in \cM_{m,n}$, there is a unique decomposition $F=U|F|$ satisfying the conditions $U \in \cI_{m,n}$ and $\ker(U)=\ker(F)$. This decomposition is called the \textit{canonical polar decomposition}.  The  partial isometry $U$ of the canonical polar decomposition is  the \textit{canonical partial isometry}. 

If the matrix $F$ has an SVD given by $F=V\Sigma W^*$ and rank equal to $r$, the canonical partial isometry can be expressed as
\begin{equation}\label{canon part iso}
U=V \begin{bmatrix}   I_r  & 0 \\  0 & 0_{m-r,n-r}      \end{bmatrix} W^*.
\end{equation}
This expression does not depend on the pair of matrices $V$, $W$ that one considers for an SVD. 

\begin{rem}\label{cdp uniq} The requirement that $\ker(U)=\ker(F)$ in the canonical polar decomposition forces $U$ to be unique. When $r=q \, (=\min\{ m , \, n \})$, there is also a unique partial isometry $X$ such that $F=X|F|$. Indeed,  $X$ is given by the  formula (\ref{canon part iso}) after deleting all the zero rows in the second factor when $q=n$ or all the zero columns when $q=m$.  
In the case in which $r<q$ there are infinitely many partial isometries $X$ satisfying $F=X|F|$. It is well known that all these  partial isometries are 
described by
$$
X=V \begin{bmatrix}   I_r  & 0 \\  0 & S      \end{bmatrix} W^*.
$$
where $S \in \cM_{m-r,n-r}$ is a partial isometry (see for instance \cite[Thm. 8.1]{H08} or \cite{HJ}). 
\end{rem}

\medskip

\noi \textbf{Majorization, unitarily invariant norms and symmetric gauge functions.}
Let $x=(x_1, \ldots, x_q)$ be a vector in $\R^q$. We denote by $x^{\downarrow}$ the vector obtained by rearranging
the coordinates of $x$ in  nonincreasing order. This means that  
$x^{\downarrow}=(x_1^{\downarrow}, \ldots, x_q^{\downarrow})$ satisfies $x^{\downarrow}_1\geq \ldots \geq x^{\downarrow}_q$.
Given two vectors $x,y \in \R^q$, we write $x \smayo_w  y$ if 
\begin{equation}\label{def may}
\sum_{i=1}^k x^{\downarrow}_i \leq \sum_{i=1}^k y^{\downarrow}_i , \, \, \, \, \, \, \, \,  k=1, \ldots,q. 
\end{equation}
In this case, we say that $x$ is \textit{submajorized} by $y$ and we write $x\smayo_w y$. If, in addition, there is an equality for $k=q$ in (\ref{def may}), then we write  $x \smayo y$ and  we say that $x$ is \textit{majorized} by $y$. 

A norm $\Phi$ on $\R^q$ is called a \textit{symmetric gauge function} if it satisfies  
$$
\Phi(x_1,  \ldots , x_q)=\Phi(|x_{\sigma(1)}|,  \ldots , |x_{\sigma(q)}|),
$$
for every permutation $\sigma$ of the integers $1,\ldots, q$. Symmetric gauge functions can be used to define norms in $\cM_{m,n}$. Indeed, if $q=\min\{m,n\}$ and $\Phi$ is a symmetric gauge function on $\R^q$, one can define a  norm on $\cM_{m,n}$ by
\begin{equation}\label{nuis in lch}
\| F\|_\Phi = \Phi(s_1(F), \ldots , s_q(F)).
\end{equation}
From this expression in terms of the singular values, these norms  turn out to be unitarily invariant; that is,
$
\|WFU\|_\Phi =\|F\|_\Phi \,, 
$
whenever $F\in \cM_{m,n}$,  $U \in \cU_{n}$ and  $W \in \cU_{m}$. Conversely, every unitarily invariant norm in $\cM_{m,n}$ can be constructed as in formula (\ref{nuis in lch})  for some symmetric gauge function (see e.g. \cite[Thm. 3.5.18]{HJ91}).

Many well-known norms are unitarily invariant norms. If the symmetric gauge  function is the $\ell^p$ norm, for $p\geq 1$, then one gets the \textit{$p$-Schatten norms} on the space $\cM_{m,n}$. If $p=\infty$, then we get the usual  operator norm. If $p=2$, then we have the \textit{Frobenious norm}, which has the following expressions
$$
\| F \|_2 = \left( \sum_{i=1}^q s_i ^2(F) \right )^{1/2}=\left(\sum_{i=1}^n \| F v_i\|^2\right)^{1/2}; 
$$
where $\{ v_i \}_1^n$ is any orthonormal basis of $\C^n$.
Another relevant norms are the \textit{Ky-Fan $k$-norms} ($1\leq k \leq q$) given by 
$$
\| F\|_{(k)}=\sum_{i=1}^k s_i(F),
$$
where the symmetric norms are $\Phi_{(k)}(x_1, \ldots , x_q)=\sum_{i=1}^k |x|_i^{\downarrow}$.

The connection between the notion of majorization and the unitarily invariant norms is the following result, known as the Ky-Fan dominance principle:

\begin{teo}\label{KF dominance}
Let $x,y\in\R^q$,  then the following statements are equivalent:
\begin{itemize}
\item[i.)] $\Phi(x)\leq \Phi(y)$ for every symmetric gauge function $\Phi$;
\item[ii.)] $x\smayo_w y$.
\end{itemize}
\end{teo}





In order to prove uniqueness properties, we will restrict to the following particular class of symmetric gauge functions. A norm $\Phi$ on $\R^q$ is \textit{strictly convex}  when for all $x,y \in \R^q$, $x\neq y$, if $\Phi(x)=\Phi(y)=1$ and $\lambda \in (0,1)$, then 
$$
\Phi(\lambda x+(1- \lambda) y)< 1.
$$
For instance, Schatten $p$-norms are strictly convex if $p \in (1,\infty)$. On the other hand,  Schatten $p$-norms with $p=1,\infty$, and Ky-Fan norms are not strictly convex.

\section{Approximation by partial isometries}\label{main}

Let $\|\cdot\|_\Phi$ be a unitarily invariant norm in $\cM_{m,n}$, where $\Phi$ is a symmetric gauge  function in $\R^q$. As before, we set $q=\min\{m,n\}$. We begin by stating the solution to the best approximation problem of a given matrix $F\in\cM_{m,n}$ by partial isometries in $\cI_{m,n}^k$ with respect to the distance induced by $\|\cdot\|_\Phi$.

\begin{teo}\label{teo 1 matrix}
Let $1\leq k \leq q$, and let $F \in \cM_{m,n}$. Given a singular value decomposition $F=V\Sigma W^*$, the partial isometry
\begin{equation}\nonumber
U_k=V \begin{bmatrix}   I_k & 0  \\   0 & 0_{m-k,n-k}      \end{bmatrix}  W^*
\end{equation}
belongs to $\cI_{m,n}^k$ and satisfies 
\begin{equation}\label{min rank lower}
\| F- U_k\|_\Phi = \min \{ \,  \| F-X \|_\Phi    \, : \,    X \in \cI_{m,n}^k \, \},
\end{equation}
for every unitarily invariant norm $\|\cdot\|_\Phi$.
\end{teo}

\begin{rem}\label{formulas d}
From the above result, we can compute the distance of a matrix $F$ of rank $r$ to each connected component $\cI_{m,n}^k$ of the partial partial isometries, i.e.
$$
d_\Phi(F,\cI_{m,n}^k)=\min \{ \, \| F- X \|_\Phi \, : \,X \in \cI_{m,n}^k  \, \},
$$
where $k=1,\ldots, q$. Indeed, we have
\begin{itemize}
\item If $k<r$:  $d_\Phi(F,\cI_{m,n}^k)=\Phi(s_1(F)-1,\ldots, s_k(F)-1, s_{k+1}(F), \ldots, s_r(F),0,\ldots,0)$.
\item If $k=r$:  $d_\Phi(F,\cI_{m,n}^r)=\Phi(s_1(F)-1,\ldots, s_r(F)-1, 0,\ldots,0)$.
\item If $k>r$:  $d_\Phi(F,\cI_{m,n}^k)=\Phi(s_1(F)-1,\ldots, s_r(F)-1, \underbrace{1, \ldots, 1}_{k-r}, 0,\ldots,0)$.
\end{itemize}
\end{rem}

\begin{teo}[Uniqueness]\label{teo 1 sub 1 matrix}
Let $1\leq k \leq q$, and let $F \in \cM_{m,n}$. Suppose that   $\Phi$ is a strictly convex symmetric gauge function. Then every minimizer of the above problem  \eqref{min rank lower} has the form
\begin{equation}\nonumber
U_k=V \begin{bmatrix}   I_k & 0  \\   0 & 0_{m-k,n-k}      \end{bmatrix}  W^* ,
\end{equation}
where  $V$ and $W$ are any pair of unitary matrices such that $F=V\Sigma W^*$ is an SVD.
\end{teo}

It is not difficult to construct examples which show that there are other minimizers when  $\Phi$ is not strictly convex. 

\begin{exa}
Given $a>b>1$, consider the matrices
$$
F=\begin{bmatrix}a&0\\0&b\end{bmatrix},\quad U_1=\begin{bmatrix}1&0\\0&0\end{bmatrix},\peso{and} X=\begin{bmatrix}0&0\\0&1\end{bmatrix}.
$$
Take the Schatten norm for $p=1$, i.e.  $\| A \|_1=s_1(A) + s_2(A)$, $A \in \cM_2$. If we look for the closest partial isometry of rank one, then
$$
\| F- U_1 \|_1=a+ b-1 =\| F- X\|_1.
$$
From the assumption on the numbers $a$, $b$, it follows that all the possible unitaries associated to the SVDs of $F$ are given by  $V=W=\text{diag}(\lambda_1,\lambda_2)$, $|\lambda_1|=|\lambda_2|=1$. 
Therefore, $X$ is a minimizer which does not have the  form described in Theorem \ref{teo 1 sub 1 matrix}.
\end{exa}

The expression of the minimizers can be further simplified according to the following cases.

\begin{teo}[Parametrization of minimizers]\label{parametrization of min}
Let $1\leq k \leq q$, and let $F \in \cM_{m,n}$ be of rank $r$. Suppose that   $\Phi$ is a strictly convex symmetric gauge function.
Then the minimizers of problem  \eqref{min rank lower} satisfy the following:  
\begin{enumerate}
\item[i)] If $k<r$, then there is a unique minimizer if and only if $s_k(F)\neq s_{k+1}(F)$. In the case in which $s_k(F)= s_{k+1}(F)$, there are infinitely many minimizers given as follows. Set 
$$
\ell_k=\#\{j:\ s_j(F)< s_k(F) \}, \, \, \, \, \, \, \,  e_k=\# \{ \, j  \, : \, s_j=s_k \, \} .
$$ 
Given  $F=V\Sigma W^*$ an SVD,   the minimizers are parametrized by
$$
U_{k,P}=V\begin{bmatrix}  I_{\ell_k} & 0 & 0 \\ 0  & P  & 0 \\ 0 & 0 & \hspace{0.5cm}0_{m-\ell_k - e_k,n-\ell_k - e_k}  \end{bmatrix}W^*,
$$
where   $P$ is an orthogonal projection in $\cM_{e_k}$ of rank $k-\ell_k$.
\item[ii)] If $k=r$  and $F=U|F|$ is the canonical polar decomposition, then $U$
 is the unique minimizer.
\item[iii)] If $r<k\leq q$, then there are infinitely many minimizers.  Given  $F=V\Sigma W^*$ an SVD,  the minimizers are described by
$$
U_{k,S}=V \begin{bmatrix}   I_r & 0  \\   0 & S      \end{bmatrix} W^* ,
$$
where $S$ is a partial isometry of rank $k-r$.
\end{enumerate}
\end{teo}

\begin{rem}\label{canonical of best rank app}
Let $F=V\Sigma W^*$ be an SVD, and suppose that $1 \leq k \leq r$. Define the orthogonal projection
$$
P_k=W \begin{bmatrix} I_k & 0 \\ 0 & 0 \end{bmatrix}W^*.
$$
Noting that the non-zero singular values of $FP_k$ are $s_1(F), \ldots, s_k(F)$, we see that an SVD is given by
$$
FP_k=V \begin{bmatrix} diag(s_1(F), \ldots, s_k(F)) & 0 \\ 0 & 0_{m-k,n-k} \end{bmatrix} W^*.
$$
If $FP_k=U_0|FP_k|$ is the canonical polar decomposition,  then from the   expression in (\ref{canon part iso}), we find that 
$U_0=U_k.$ As a consequence of our results, and following the terminology in \cite{CFP03},  any solution to the problem of finding the closest partial isometry with a given lower or equal rank can be constructed by a lift-and-project method. First, we lift the problem to the space $\cM_{m,n}$ and find a closest matrix of  lower or equal rank given by $FP_k$. Then, we project to the set of all partial isometries by taking the partial isometry in the canonical polar decomposition.

In the case where $r<k\leq q$, we cannot find the minimizers using a lift-and-project method. Indeed,
\begin{equation}\nonumber
 \inf\{ \,  \|F-G\|_\Phi  \, : \, G \in \cM_{m,n}\, , \, \rank(G) = k  \, \}
\end{equation} 
is attained if and only if $\rank(F)\leq k$ (see \cite[Thm. 3]{M60}). However, the set of minimizers $\{ \, U_{k,S} \, : \, S \in \cI_{m-r,n-r}^{k-r} \, \}$ has another characterization. It consists of all the partial isometries $X$ of rank $k$ satisfying $F=X|F|$ (see Remark \ref{cdp uniq}).
\end{rem}

\bigskip

As a consequence  of the above results we consider the global best approximation problem: given $F \in \cM_{m,n}$, we seek for the minimizers of 
\begin{equation}\label{eq global 1}
d_\Phi(F,\cI_{m,n})=\min \{ \, \| F- X \|_\Phi \, : \,X \in \cI_{m,n}  \, \}.
\end{equation}
Using that the sets $\cI_{m,n}^k$, $k=1, \ldots, q$, are the connected components of $\cI_{m,n}$,   we see that 
$$
d_\Phi(F,\cI_{m,n})=\min_{1 \leq k \leq r} d_\Phi(F,\cI_{m,n}^k), 
$$ 
where  $r=\rank(F)$. Notice that we have ruled out the components given by $r < k \leq  q$. This follows from  Remark \ref{formulas d}, which implies  that $d_\Phi(F,\cI_{m,n}^r)\leq d_\Phi(F,\cI_{m,n}^k)$ for $k=r+1, \ldots, q$. Then, global minimizers can be obtained from the connected components $\cI_{m,n}^k$, $k=1, \ldots, r$. Thus, there are also characterized as the canonical partial isometry associated to a best lower or equal rank approximation. The connected component where each minimizer lies is determined by the singular values of the matrix $F$. We will omit the proof of our next result, which now follows from these remarks, Theorems \ref{teo 1 matrix} and  \ref{teo 1 sub 1 matrix}.

\begin{teo}[Global minimizers]\label{teo 3 matrix}
Let $F \in \cM_{m,n}$ be of rank $r$.
\begin{enumerate}
\item[a)] Suppose that $s_i(F)\neq 1/2$ for all $i=1, \ldots, r$. 
\begin{enumerate}
\item[i)] If $s_r(F)>1/2$ and $F=U|F|$ is the canonical partial decomposition, then $U$ is a minimizer of \eqref{eq global 1}.

More generally, if $s_{k}(F)>1/2 >s_{k+1}(F)$ for some $k=1, \ldots, r-1$, then the closest partial isometry from $\cI_{m,n}^k$ to $F$ is a minimizer of \eqref{eq global 1}. 
\item[ii)]  If $s_1(F) <1/2$, then the closest partial isometries from $\cI_{m,n}^1$ to $F$ are minimizers of \eqref{eq global 1}. 
\end{enumerate}
\item[b)] Suppose that 
 $\{  j : \, s_j(F)=1/2  \}=\{ k, \ldots ,k +l \}$, where $1 \leq k \leq r$ and $0\leq l \leq r-k$. 
\begin{enumerate}
\item[i)] If $k=1$, then the closest partial isometries from $\,\cI_{m,n}^1,\,\ldots, \,\cI_{m,n}^{l+1}$ to $F$ are minimizers of \eqref{eq global 1}.
\item[ii)] If $1<k\leq r$, then the closest partial isometries from $\,\cI_{m,n}^{k-1},\,\ldots, \,\cI_{m,n}^{k+l+1}$ to $F$ are minimizers of \eqref{eq global 1}.
\end{enumerate} 
\end{enumerate}
Moreover, in all the cases  these are the only minimizers provided that the symmetric gauge function $\Phi$ is strictly convex.
\end{teo}



\section{Symmetric approximation of finite frames}\label{sec frames}

Let $\cH$  be an $m$-dimensional Hilbert space.
 A family of vectors $\{  f_j \}_1^n$ spanning a subspace $\cK\subseteq \cH$  is called a
\textit{finite frame} for $\cK$. Equivalently, $\{  f_j \}_1^n$ is a finite frame for $\cK$ if there are constants $A,B>0$ such that,  for every $f \in \cK$, we have 
\[   A\|f\|^2 \leq \sum_{j =1}^n |\PI{f}{f_j}|^2  \leq B \| f\|^2 . \] 
 The optimal constants  $A$, $B$ where these inequalities hold for all $f \in \cK$ are called the lower and upper bounds for the frame. The  frame is a \textit{tight frame} if $A=B$ and a \textit{Parseval frame} if $A=B=1$.  A comprehensive introduction to frame theory in finite dimensional spaces and its applications  can be found in  \cite{CK12, [Chrbook]}. As in the Introduction, we will assume that 
$\cH=\C^m$. There is no loss of generality because the Frobenious norm is unitarily invariant and our results  only depend on the coordinates of vectors.

For the sake of simplicity,  finite frames (resp. Parseval frames) for subspaces of $\C^m$ consisting of $n$ vectors, counted with repetitions if it is necessary, will be called $n$-frames (resp. $n$-Parseval frames).

\begin{fed}
Let  $\cJ$ be a family of $n$-Parseval frames. Given a frame $\{ f_j\}_1^n$ for a  subspace $\cK \subseteq \C^m$,  a Parseval frame $\{ u_j \}_1^n \in \cJ$  is called a \textit{symmetric approximation of $\{ f_j\}_1^n$ in $\cJ$} if 
$$
\sum_{j=1}^n \| f_j - u_j\| ^2 \leq \sum_{j=1}^n \| f_j - x_j \|^2
$$
for all Parseval frames $\{ x_j \}_1^n \in \cJ$.
\end{fed}

\begin{rem}\label{weakly e}
Let $\cJ_w$ be the family of all  $n$-Parseval frames weakly similar to a given frame  $\{ f_j\}_1^n$.
 Then, 
 a Parseval frame $\{ u_j \}_1^n \in \cJ_w$ is a symmetric approximation of $\{ f_j\}_1^n$ in $\cJ_w$ exactly when $\{ u_j \}_1^n$ is a symmetric approximation of $\{ f_j\}_1^n$ in the sense of \cite{FPT02}.
\end{rem}

\bigskip

In the Introduction, we have explained the connection of approximation by partial isometries and symmetric approximation of frames. The first family of $n$-Parseval frames that we shall consider to study symmetric approximations is determined by the rank of the partial isometries.
In the language of frames,  this  corresponds to the notion of redundancy or excess of a frame. To recall this notion, we write $I_n$ for the set $\{ \, 1, \ldots , n\}$. The \textit{excess} of a frame $\{ f_j \}_1^n$ for a subspace $\cK$ is 
\[   e(\{ f_j \}_1^n)=\max \{ \, |I|  \, : \, I\subseteq I_n, \, \text{span}\{ f_j \}_{j \in I_n\setminus I} = \cK\, \}=n - \dim \cK.  \]
It describes the greatest number of vectors which can be removed from $\{ f_j \}_1^n$ with the property that the remaining vectors
still generate the same subspace. It is easily verified that $e(\{ f_j \}_1^n)=\dim \ker(F)$, where $F$ is the synthesis matrix of $\{ f_j \}_1^n$. Recall that $F$ is the matrix whose columns are the vectors $f_1, \ldots, f_n$.
In particular, this implies $n=e(\{ f_j \}_1^n)+\rank(F)$, and consequently, $n-q\leq e(\{ f_j \}_1^n)\leq n-1$, $q=\min\{m,\,n\}$.
We  introduce the following  sets of $n$-Parseval frames: for $k=1, \ldots, q$, 
$$
\cJ_{n-k}=\{ \, \{ w_j\}_1^n  \, : \, \{ w_j\}_1^n \text{ is a Parseval frame},  \,e(\{ w_j \}_1^n)=n-k \, \}.
$$
Theorems \ref{teo 1 matrix}, \ref{teo 1 sub 1 matrix} and \ref{parametrization of min} can be rephrased as results providing all the symmetric approximations in $\cJ_{n-k}$.
For the uniqueness part, we recall that the Frobenious norm is given by a strictly convex  symmetric gauge function. 

\begin{cor}\label{fixed ex}
Let $\{ f_j\}_1^n$ be a frame for a subspace $\cK \subseteq \C^m$. Suppose that this frame has synthesis matrix $F$ of rank $r$.  Then every symmetric approximation of $\{ f_j\}_1^n$ in $\cJ_{n-k}$, $k=1, \ldots, q$, is given by
$$
u_j=V\begin{bmatrix} I_k & 0 \\ 0 & 0_{m-k,n-k} \end{bmatrix}W^* e_j, \, \, \, \, j=1, \ldots, n, 
$$
where $V$ and $W$ are any pair of unitary matrices such that $F=V\Sigma W^*$ is an SVD.
In particular, there is a unique symmetric approximation  if and only if $s_k(F)\neq s_{k+1}(F)$ when $k<r$; the canonical Parseval frame is the unique symmetric approximation  when $k=r$; and there are infinitely many symmetric approximations  when $r<k\leq q$. 
\end{cor}

\begin{rem}
Let $\{ f_j\}_1^n$ and $\{ g_j\}_1^n$ be frames for subspaces of $\C^m$ whose synthesis matrices are $F$ and $G$, respectively. 
Notice that these frames are weakly equivalent if and only if $\ker(F)=\ker(G)$. Then any pair of weakly equivalent frames have the same excess. This means that we have the inclusion $\cJ_w \subseteq \cJ_{n-r}$. Thus,  Corollary \ref{fixed ex} generalizes 
\cite[Thm. 1.3]{FPT02}. The fact that the canonical Parseval frame is the unique symmetric approximation in the family  $\cJ_{n-r}$ was also proved in \cite[Corol. 3.7]{H2} and \cite[Thm. 2.2]{LZ06}. 
\end{rem}

 We give below a somewhat simplified version of Theorem \ref{teo 3 matrix} for  frames. The approximation of a matrix by a partial isometry without rank constraints corresponds to the symmetric approximation of a frame in the family of all $n$-Parseval frames. We will only exhibit one symmetric approximation from each connected component of the partial isometries. 

\begin{cor}\label{frames global}
Let $\{ f_j\}_1^n$ be a frame for a subspace $\cK \subseteq \C^m$ with synthesis matrix $F$. Set  $k= \# \{ \, j \, : s_j(F) \geq 1/2 \, \}$, and put $k=1$ if $s_j(F)<1/2$ for all $j\geq 1$.  Then a symmetric approximation of $\{ f_j\}_1^n$ in the family of all $n$-Parseval frames is given by
$$
u_j=V\begin{bmatrix} I_k & 0 \\ 0 & 0_{m-k,n-k} \end{bmatrix}W^* e_j, \, \, \, \, j=1, \ldots, n, 
$$
where $V$ and $W$ are any pair of unitary matrices such that $F=V\Sigma W^*$ is an SVD. 
This is the unique symmetric approximation in the following two cases: $(i)$ there is at least one singular value satisfying $s_j(F)\geq 1/2$ and $\# \{ \, j \, : s_j(F) = 1/2 \, \} \leq 1$; or
$(ii)$ $s_j(F)<1/2$ for all $j\geq 1$ and $s_1(F)\neq s_2(F)$.
\end{cor} 

Note that the above symmetric approximation is the canonical Parseval frame associated to $\{ f_j\}_1^n$ when $k=\rank(F)$, and thus, it spans the subspace $\cK$.  In general, a symmetric approximation in the family of all $n$-Parseval frames   spans a subspace contained in $\cK$. Using this observation we can now find a symmetric approximation in the family $\cJ_\cS$ of all the $n$-Parseval frames whose vectors belong to a fixed subspace $\cS\subseteq \C^m$.  In the following result, we denote by $P_\cS$ the orthogonal projection onto $\cS$. 



\begin{teo}\label{fix the subspace}
Let $\{ f_j\}_1^n$ be a frame for a subspace $\cK \subseteq \C^m$ with synthesis matrix $F$.
Fix a subspace $\cS \subseteq \cH$ such that $\cK$ is not included  in  $\cS^\perp$. Set  $k= \# \{ \, j \, : s_j(P_\cS F) \geq 1/2 \, \}$, and put $k=1$ if $s_j(P_\cS F)<1/2$ for all $j\geq 1$.  Then a symmetric approximation of $\{ f_j\}_1^n$ in $\cJ_\cS$ is given by
$$
u_j=V\begin{bmatrix} I_k & 0 \\ 0 & 0_{m-k,n-k} \end{bmatrix}W^* e_j, \, \, \, \, j=1, \ldots, n, 
$$
where $V$ and $W$ are any pair of unitary matrices such that $P_\cS F=V\Sigma W^*$ is an SVD. 

\end{teo}

\begin{rem}
$i)$ Clearly, the same uniqueness statement of Corollary \ref{frames global} holds replacing $F$ by $P_\cS F$.

$ii)$ If we have $\cK \subseteq \cS^\perp$, then $P_\cS F=0$. Then the above formula does not give a symmetric approximation in $\cJ_\cS$. However, in this case, it is easily seen that every  $n$-Parseval frame in $\cS$ is a symmetric approximation.  
\end{rem}

\section{Proofs}\label{proofs}

\subsection{Proof of Theorem \ref{teo 1 matrix}}

 We are going to use the following version of the Lidskii-Mirsky-Wielandt theorem stated for singular values (see \cite{M60}, \cite[Thm. IV.3.4]{B97} and for a proof for rectangular matrices \cite[Thm. 3.4.5]{HJ91}).

\begin{teo}\label{mirsky}
Let $F,G \in \cM_{m,n}$ and $q=\min\{  m, \, n \}$,  then
$$
(|s_1(F)-s_1(G)|, \ldots , |s_q(F)-s_q(G)| ) \,   \textstyle{\smayo_w}  \, s(F-G).
$$
\end{teo}

\medskip

 Now we prove  Theorem \ref{teo 1 matrix}. Let $X$ be a partial isometry of rank $k$. 
Note that $s_j(X)=1$ for $j\leq k$ and $s_j(X)=0$ for $j>k$.
Using  Theorem \ref{mirsky}, we have the submajorization
\begin{equation}\nonumber
(|s_1(F)-1|, \ldots, |s_k(F)-1|, s_{k+1}(F),  \ldots, s_q(F)) \,   \textstyle{\smayo_w}  \, s(F-X).
\end{equation}
According to the Ky-Fan dominance principle (Theorem \ref{KF dominance}), this implies that 
\begin{equation}\label{ineq teo 1}
\Phi(|s_1(F)-1|, \ldots, |s_k(F)-1|, s_{k+1}(F),  \ldots, s_q(F)) \leq \| F- X \|_\Phi
\end{equation}
for every symmetric gauge function $\Phi$. 

On the other hand,  the partial isometry $U_k$  clearly satisfies $\rank(U_k)=k$. From its definition, we see that 
$$
F-U_k=V \left(\Sigma - \begin{bmatrix}  I_k   &   0 \\  0  & 0_{m-k,n-k} \end{bmatrix} \right)W^*.
$$
Therefore, 
\begin{equation}\label{lid eq}
s(F-U_k)= (|s_1(F)-1|, \ldots, |s_k(F)-1|, s_{k+1}(F),  \ldots, s_q(F))^{\downarrow}.
\end{equation}
This fact along with the inequality \eqref{ineq teo 1} gives
$$
\| F-U_k \|_\Phi \leq \| F- X\|_\Phi.
$$ 
Hence the partial isometry $U_k$ is a minimizer.  

\subsection{Proof of Theorem \ref{teo 1 sub 1 matrix}}

The proof is divided into four steps. 
The first step contains an uniqueness property of strictly convex symmetric gauge functions. In the second step, we give a lemma on the equality case in the Lidskii-Mirsky-Wielandt theorem
stated for singular values. In the  third step, we prove an optimization lemma for numbers. 
These results  are combined in the fourth step to prove Theorem \ref{teo 1 sub 1 matrix}.

\bigskip

\noi \textit{Step 1.} Strictly convex symmetric gauge functions have the following property. 

\begin{lem}\label{prop strict conv}
Let $x,y \in \R^q$, $x\geq 0$, $y\geq 0$. Suppose that $\Phi$ is a strictly convex symmetric norm,  $\Phi(x)=\Phi(y)$ and  $x \textstyle{\smayo_w} y$. Then $x=y$.
\end{lem}
\begin{proof}
For vectors $x\geq 0$, $y\geq 0$, we have $x \textstyle{\smayo_w} y$ if and only if $x$ is in the convex hull of the $2^qq!$ vectors given by
$$
y_{j}:=(\eps_{j,1} x_{\sigma_j(1)}, \ldots, \eps_{j,q} x_{\sigma_j(q)}),
$$ 
where   $(\eps_{j,1}, \ldots, \eps_{j,q})\in\{-1,1\}^q$  and $\sigma_j$ is a permutation of $\{1,\ldots,q\}$ for every $j$ (see \cite[Ex. II.2.10]{B97}). Thus, we can write $x=\sum \lambda_j\, y_j$, where $\sum \lambda_j=1$ and $\lambda_j\geq 0$.
Now note that $\Phi(y_j)=\Phi(y)=\Phi(x)$ and 
$$
\Phi(x)=\Phi\left(\sum \lambda_j\, y_j\right)\leq \sum \lambda_j\, \Phi (y_j)
= \sum \lambda_j \, \Phi (x)=\Phi(x).
$$
This implies that $\Phi\left(\sum \lambda_j\, y_j\right) = \sum \lambda_j\, \Phi (y_j)$. 
Since the norm $\Phi$ is strictly convex and $\Phi(y_j)=\Phi(y)$ for all $j$, we get that $y_j=y$ whenever $\lambda_j>0$. Hence $x=y$.  
\end{proof}

\bigskip

\noi \textit{Step 2.} Let  $A \in \cM_n$ be a Hermitian matrix. The vector of eigenvalues of $A$ arranged in nonincreasing order is denoted by
$\lambda(A)=(\lambda_1(A), \ldots, \lambda_n(A))$. The  Lidskii-Mirsky-Wielandt theorem states:

\begin{teo}\label{first lid}
Let $A,B \in \cM_n$ be Hermitian matrices. Then 
$$
\lambda(A+B) - \lambda(A) \, \textstyle{\smayo}\, \lambda(B).
$$
\end{teo}

We refer the reader to the book by Bhatia \cite{B97} for three different proofs and historical information. 
However, the simplest proof  was given later  in \cite{LM99}. Recently, the same ideas of this last proof were used in \cite{MRS14} to investigate   the equality case in the Lidskii-Mirsky-Wielandt theorem. In particular, the following was proved:  

\begin{pro}\label{eq in Lid hermitian}
Let $A,B \in \cM_n$ be Hermitian matrices. Then $$\left(\lambda(B+A)-\lambda(A)\right)^{\downarrow}=\lambda(B)$$ implies $AB=BA$.
\end{pro}

Note that Theorem \ref{mirsky} can be deduced from Theorem \ref{first lid}.  The trick is to consider the Hermitian matrices
$$
\widehat{F}=\begin{bmatrix} 0 & F \\ F^* & 0 \end{bmatrix}, \, \, \, \, \, 
\widehat{G}=\begin{bmatrix} 0 & G \\ G^* & 0 \end{bmatrix}. 
$$
Using the same trick we get the following result.

\begin{lem}\label{rem lidsk eq sing val}
Let $F,G \in \cM_{m,n}$. Then 
$$
s(F-G)=| s(F) - s(G) |^{\downarrow}
$$
implies that $GF^*=FG^*$ and $G^*F=F^*G$. 
\end{lem}
\begin{proof}
Let $\widehat{F},\widehat{G}\in\cM_{m+n}$ be the matrices defined above. Note that  $B=\widehat{F}-\widehat{G}$ and $A=\widehat{G}$ are Hermitian matrices  satisfying 
$$
\lambda(A)=(s(G), -s(G)^\uparrow)\peso{and} \lambda(B)=(s(F-G),-s(F-G)^\uparrow),
$$
where $s(G)^\uparrow$ and $S(F-G)^\uparrow$ are the vectors of singular values of $G$ and $F-G$ arranged in nondecreasing order. From the assumed equality of the singular values we obtain
$$
\left(\lambda(B+A)-\lambda(A)\right)^{\downarrow}=\lambda(B).
$$
Proposition \ref{eq in Lid hermitian} implies that $AB=BA$. This is equivalent to $\widehat{F}\widehat{G}=\widehat{G}\widehat{F}$, which  means that $GF^*=FG^*$ and $G^*F=F^*G$. 
\end{proof}

\bigskip

\noi \textit{Step 3.} Now we prove the following optimization result. 

\begin{lem}\label{lem prob nros}
Given $s\in\R^q$ such that $s_1 \geq \ldots \geq s_q \geq 0$, let $f:\{-1,0,1\}^q\to\R$ be the function defined by
$$
f(x_1, \ldots, x_q)=\sum_{j=1}^q (s_j - x_j)^2. 
$$ 
Set 
$$
r=\max \{ \, j \, : \, s_j \neq 0\, \}.
$$
For every $1 \leq k \leq q$,  the minimizers of $f$  subject to the restriction $\# \{ \, j \, : x_j \neq 0 \, \}=k$ have the following structure:
\begin{enumerate}
\item[i)] If $1\leq k <r$ and  $s_k \neq s_{k+1}$, then there is an unique minimizer  given by
$$
x_1= \ldots=x_k=1, \, x_{k+1}=\ldots = x_{q}=0.
$$ 
If $s_k = s_{k+1}$, set $\ell_k=\# \{\, j \, : \, s_j < s_k \, \}$ and $e_k=\# \{ \, j  \, : \, s_j=s_k \, \}$.
Then there are $\binom{e_k}{k-\ell_k}$ minimizers given by
\begin{align*}
& x_1= \ldots =x_{\ell_k}=1,  \  x_{\ell_k + 1}=\sigma_1, \ldots , x_{\ell_k + e_k}=\sigma_{e_k}, \ x_{\ell_k+ e_k + 1}= \ldots =x_q=0,
\end{align*}
where each $\sigma_i \in \{ \, 0, \, 1 \, \}$,  and $\# \{ \, i \, : \, \sigma_i=1 \, \}=k- \ell_k$.
\item[ii)] If $k=r$, then there is an unique minimizer  given by
$$
x_1= \ldots=x_r=1, \, x_{r+1}=\ldots x_q=0.
$$ 
\item[iii)] If $r < k \leq q$, then there are $2^{k-r} \binom{q-r}{k-r}$ minimizers given by 
\begin{align*}
x_1= \ldots =x_r=1,  \  x_{r + 1}=\sigma_1, \ldots , x_q=\sigma_{q-r}\, ,
\end{align*}
where each $\sigma_i \in \{ \, -1, \, 0, \, 1 \, \}$,  and $\# \{ \, i \, : \, \sigma_i \neq 0 \, \}=k- r$.
\end{enumerate}
\end{lem}
\begin{proof}
To prove $i)$ we first note that any minimizer must satisfy $x_i=0$, $i=r+1, \ldots, q$. This follows immediately using the inequality $(a-1)^2< a^2 +1$ for $a>0$. From the inequality $(a-1)^2<(a +1)^2$ for $a>0$, we deduce that  any minimizer satisfies $x_i \neq -1$, $i=1, \ldots ,r$.  Given two  numbers $ a \geq b$, then
$$
(a-1)^2 + b^2 \leq a^2 + (b-1)^2,
$$
and  equality holds if and only if $a=b$. Using this elementary inequality recursively, it is easy to see that a minimizer is obtained by taking $x_1=\ldots=x_k=1$ when $s_k\neq s_{k+1}$. On the other hand, it is clear that if $s_k=s_{k+1}$, the minimizer is not unique. Indeed, as before, any minimizer must have $x_1=\ldots =x_{\ell_k}=1$. But now $s_{\ell_k + 1}= \ldots =s_{\ell_k + e_k}=s_k$, which implies that there are  $\binom{e_k}{k-\ell_k}$ possible choices to place the remaining $k- \ell_k$ ones. 

The proofs of $ii)$ and $iii)$ are similar. We only remark that in the last item, one also has to take into account minimizers satisfying $x_i=\pm 1$ for $i>r$. 
\end{proof}

\bigskip

\noi \textit{Step 4.} Suppose that $\Phi$ is a strictly convex symmetric gauge function. Let $X$ be a partial isometry of rank $k$ such that $\|F-X\|_\Phi=\|F-U_k\|_\Phi$. We have shown in the proof of Theorem \ref{teo 1 matrix} that $s(F-U_k)  \textstyle{\smayo_w} s(F-X)$. According to Lemma \ref{prop strict conv}, it follows that $s(F-U_k)=s(F-X)$. Then, using the equality (\ref{lid eq}) we get
\begin{equation}\label{eq in lidskii sing v}
s(F-X)=| s(F) - s(X) |^{\downarrow}. 
\end{equation}
Now Lemma \ref{rem lidsk eq sing val} gives $XF^*=FX^*$ and $X^*F=F^*X$.  From these latter relations, one can prove that there exist  $V_0 \in \cU_m$ and $W_0 \in \cU_n$ such that $F=V_0 D_F W_0^*$ and $X=V_0 D_X W_0^*$, where $D_F, D_X \in \cM_{m,n}$ are diagonal matrices with real coefficients. Furthermore, $D_F$ can be taken to be with non negative coefficients (see \cite[Thm. II]{EY36}).

Since $|F|=W_0|D_F|W_0^*$ and $(D_F)_{ii}\geq 0$, we see that $s_{j_i}(F)=(D_F)_{ii}$ for some permutation $j_1, \ldots , j_q$ of the integers $1, \ldots, q$. Thus, we can find two permutation matrices $P \in \cM_m$ and $Q \in \cM_n$ satisfying $\Sigma=PD_FQ$. Put $V_1=V_0P$ and $W_1=W_0Q$.  Therefore $F=V_1\Sigma W_1^*$ is an SVD. We can also write $X=V_1 D_X' W_1^*$, where $D_X'=PD_XQ$ is diagonal. Further, note that $X^*X=W(D_X')^*D_X'W^*$ is a projection, so  its eigenvalues are $0$ and $1$, and thus, $(D_X')_{ii} \in \{ \, -1, \, 0, \, 1  \,\}$.

Then, suppose that $L \in \cM_{m,n}$ is a diagonal matrix such that $L_{ii} \in \{ \, -1, \, 0, \, 1 \, \}$ and $\# \{ \, i \, : \, L_{ii} \neq 0 \}=k$. So we have that $Y=V_1 L W_1^*$ is a partial isometry of rank $k$. Since $s(F-X)=s(F-U_k)$, the partial isometry $X$ is a minimizer for every unitarily invariant norm. In particular, we can use the Frobenious norm:  
\begin{align}
\sum_{i=1}^q (s_i(F) - (D_X')_{ii})^2  & = \| \Sigma - D_X'\|_2 ^2 = \| F- X\|_2^2  \leq \|F- Y\|_2^2 \nonumber \\
&  = \| \Sigma - L\|_2^2 = \sum_{i=1}^q (s_i(F) - L_{ii} )^2 . \label{eq diag 1}
\end{align}	
We have to consider three cases according to Lemma \ref{lem prob nros}. Note that $r=\rank(F)$. We first assume that $k<r$ and $s_k(F)\neq s_{k+1}(F)$ to obtain
$$
(D_X')_{11}= \ldots = (D_X ')_{kk}=1, \, (D_X')_{k+1 \, k+1}= \ldots =(D_X')_{qq}=0. 
$$ 
Hence,
$$
X=V_1 \begin{bmatrix}  I_k & 0 \\ 0 & 0_{m-k,n-k}\end{bmatrix} W_1 ^*.
$$
If  $s_k(F)=s_{k+1}(F)$, then 
\begin{align*}
& (D_X ')_{11}= \ldots =(D_X ')_{\ell_k \,\ell_k}=1, \\ 
& (D_X ')_{\ell_k + 1 \, \ell_k + 1}=\sigma_1, \ldots , (D_X ')_{\ell_k + e_k \, \ell_k + e_k}=\sigma_{e_k}, \\
& (D_X ')_{\ell_k+ e_k + 1 \, \ell_k+ e_k + 1}= \ldots =(D_X ')_{qq}=0,
\end{align*}
where $\sigma_i \in \{ \, 0, \, 1 \, \}$, $i=1 \ldots, e_k$ and $\# \{ \, i \, : \, \sigma_i=1 \, \}=k- \ell_k$. We can get two permutation matrices $P' \in \cM_m$  and $Q' \in \cM_n$ such that 
$$  
P'D_X'Q'= \begin{bmatrix}   I_k & 0  \\   0 & 0_{m-k,n-k}      \end{bmatrix}. 
$$
Thus,
$$
X=V_1P' \begin{bmatrix}  I_k & 0 \\ 0 & 0_{m-k,n-k}\end{bmatrix} Q'W_1 ^*,
$$
where $V_1P'$, $W_1Q'$ are unitaries associated to some SVD of $F$ by Remark \ref{unitaries SVD}. Indeed, notice that $P'$ and $Q'$
interchange rows and columns corresponding to the multiplicity of the singular value $s_k(F)$.

The case where $k=r$ follows similarly. If $r<k\leq q$, then we have that 
\begin{align*}
(D_X ')_{11}= \ldots = (D_X ')_{rr}=1,    (D_X ')_{r+1 \, r + 1}=\sigma_1, \ldots , (D_X ')_{qq}=\sigma_{q-r}\, ,
\end{align*}
where each $\sigma_i \in \{ \, -1, \, 0, \, 1 \, \}$,  and $\# \{ \, i \, : \, \sigma_i \neq 0 \, \}=k- r$. Consider the unitary matrix
defined by 
$$R=
\left\{ \begin{array}{l} \text{diag}(\sgn (\sigma_1) , \ldots, \sgn (\sigma_{q-r})) , \, \,\, \, \,\, \, \,\, \,\, \,\, \, \, \,\, \,\, \,\, \, \, \,\, \,\, \,\, \, \,   \, \, \, \, \text{if $q=m$,} \\
\text{diag}(\sgn (\sigma_1) , \ldots, \sgn (\sigma_{q-r}), I_{m-q}) ,  \, \, \,\, \, \, \,\, \,\, \,\, \, \,   \, \, \, \, \text{ if $q=n$.}
\end{array} \right. 
$$
Then, the partial isometry $X$ can be expressed as 
$$
X=V_1 \begin{bmatrix}  I_r & 0 \\ 0 & R \end{bmatrix} 
\begin{bmatrix}  I_k & 0 \\ 0 & 0_{m-k,n-k}\end{bmatrix} W_1 ^*.
$$
From Remark \ref{unitaries SVD}, the unitary  $V_1 \, \text{diag}(I_r, R)$ is associated to 
some SVD of $F$. This completes the proof.

\subsection{Proof of Theorem \ref{parametrization of min}}

Let $F=V\Sigma W^*$ be an SVD. In order to prove $i)$,  we first note that by Theorem \ref{teo 1 sub 1 matrix} and Remark \ref{unitaries SVD} the minimizers have the form  
\begin{equation}\label{form uk}
U_k=V\begin{bmatrix} D & 0 \\ 0 & R_1 \end{bmatrix} 
\begin{bmatrix} I_k & 0 \\ 0 & 0_{m-k,n-k} \end{bmatrix} 
\begin{bmatrix} D^* & 0 \\ 0 & R_2 \end{bmatrix}
W^* , 
\end{equation}
where $D \in \cU_r$ is  block diagonal  with $D_{ij}=0$ if $s_i(F)\neq s_j(F)$, and $R_1 \in \cU_{m-r}$, $R_2 \in \cU_{n-r}$ are arbitrary.
Recalling that  
$$
\ell_k=\#\{j:\ s_j(F)< s_k(F) \}, \, \, \, \, \, \, \,  e_k=\# \{ \, j  \, : \, s_j=s_k \, \},
$$ 
we see that  $\ell_k + e_k=k$ when  $s_k(F)\neq s_{k+1}(F)$. Therefore any matrix $D$ as above can be written as $D=\text{diag}(D_1,D_2)$, where $D_1 \in \cU_k$ and $D_2 \in \cU_{r-k}$. Then the  expression in \eqref{form uk} reduces to 
\begin{equation}\nonumber
U_k=V \begin{bmatrix} I_k & 0 \\ 0 & 0_{m-k,n-k} \end{bmatrix} 
W^* , 
\end{equation} 
which is the unique minimizer. 

Suppose now that $s_k(F)=s_{k+1}(F)$. This gives  $0<k-\ell_k<e_k$, and the matrices $D$  may be written as 
$D=\text{diag}(D_1,D_2,D_3)$, where $D_1 \in \cU_{\ell_k}$, $D_2 \in \cU_{e_k}$ and $D_3 \in \cU_{r-\ell_k-e_k}$. Hence $U_k$ has the form
$$
U_k= V\begin{bmatrix}  I_{\ell_k} & 0 & 0 \\ 0  & D_2 \begin{bmatrix} I_{k-\ell_k} & 0 \\ 0 & 0 \end{bmatrix} D_2^* & 0 \\ 0 & 0 & 0 \end{bmatrix} W^*.
$$
Since every orthogonal projection $P \in \cM_{e_k}$ of rank $k-\ell_k$ can be expressed as 
$$
P=D_2 \begin{bmatrix} I_{k-\ell_k} & 0 \\ 0 & 0 \end{bmatrix} D_2^*
$$
for some $D_2 \in \cU_{e_k}$, this proves the desired parametrization of the minimizers. 

We can proceed analogously in item $ii)$. It follows that
\begin{equation}\nonumber
U_r=V\begin{bmatrix} D & 0 \\ 0 & R_1 \end{bmatrix} 
\begin{bmatrix} I_r & 0 \\ 0 & 0_{m-k,n-k} \end{bmatrix} 
\begin{bmatrix} D^* & 0 \\ 0 & R_2 \end{bmatrix}
W^* = V\begin{bmatrix} I_r & 0 \\ 0 & 0_{m-k,n-k} \end{bmatrix} 
W^*=U
\end{equation}
turns out to be the unique minimizer. 

To prove  item $iii)$, we compute as above
$$
U_k= V\begin{bmatrix} D & 0 \\ 0 & R_1 \end{bmatrix} 
\begin{bmatrix} I_k & 0 \\ 0 & 0_{m-k,n-k} \end{bmatrix} 
\begin{bmatrix} D^* & 0 \\ 0 & R_2 \end{bmatrix}
W^*
= 
V  \begin{bmatrix} I_r & 0 \\ 0 &  R_1 \begin{bmatrix} I_{k-r}  & 0 \\ 0 & 0  \end{bmatrix} R_2^*\end{bmatrix} 
 W^*  ,
$$
where $R_1 \in \cU_{m-r}$ and $R_2 \in \cU_{n-r}$. Note that every partial isometry $S$ of rank $k-r$ can be written as 
$S=R_1 \begin{bmatrix} I_{k-r}  & 0 \\ 0 & 0  \end{bmatrix} R_2^*$ for some unitary matrices $R_1$ and $R_2$.

\subsection{Proof of Theorem \ref{fix the subspace}}

Set
$$
U_k=V\begin{bmatrix} I_k & 0 \\ 0 & 0_{m-k,n-k} \end{bmatrix}W^* , 
$$
where $V$ and $W$ are unitary matrices such that $P_\cS F=V\Sigma W^*$ is an SVD. 
Let $X$ be the synthesis matrix of a Parseval frame in $\cJ_\ese$. If $\{ v_j\}_1^n$ is an orthonormal basis of $\C^n$, then using Pythagoras' theorem we get that
\begin{align*}
\|F-X\|^2_2&=\sum_{j=1}^n\|(F-X)v_j\|^2\\&=\sum_{j=1}^n\big(\|(\sub{P}{\ese}F-X)v_j\|^2+\|(\sub{P}{\ese^\bot}F)v_j\|^2\big)\\
&=\|\sub{P}{\ese}F-X\|^2_2+\|\sub{P}{\ese^\bot}F\|^2_2 \\
&\geq \|\sub{P}{\ese}F-U_k\|^2_2+\|\sub{P}{\ese^\bot}F\|^2_2 =\| F - U_k\|_2^2 \, . 
\end{align*}
Note that the inequality follows by Theorem \ref{teo 3 matrix}, and in the last equality we have used that $\ran(U_k)\subseteq \ran(P_\cS F)\subseteq \cS$.   
This finishes the proof. 


{\sc (Jorge Antezana)} {Departamento de de Matem\'atica, FCE-UNLP, Calles 50 y 115, 
(1900) La Plata, Argentina  and Instituto Argentino de Matem\'atica, `Alberto P. Calder\'on', CONICET, Saavedra 15 3er. piso,
(1083) Buenos Aires, Argentina.}     
                                               
\noi e-mail: {\sf antezana@mate.unlp.edu.ar}

\bigskip

{\sc (Eduardo Chiumiento)} {Departamento de de Matem\'atica, FCE-UNLP, Calles 50 y 115, 
(1900) La Plata, Argentina  and Instituto Argentino de Matem\'atica, `Alberto P. Calder\'on', CONICET, Saavedra 15 3er. piso,
(1083) Buenos Aires, Argentina.}     
                                               
\noi e-mail: {\sf eduardo@mate.unlp.edu.ar}


\begin{thebibliography}{XX}

\bibitem{B97} R. Bhatia, {\it Matrix Analysis}, Berlin-Heildelberg-New York, Springer (1997).

\bibitem{Bo} M. Bownik,  The structure of shift-invariant subspaces of $L\sp 2({\bf R}\sp n)$. J. Funct. Anal. 177 (2000), no. 2, 282--309.
\bibitem{CP}  Cabrelli C.,  Paternostro V., Shift-invariant spaces on LCA groups. J. Funct. Anal. 258 (2010), no. 6, 2034-2059.


\bibitem{CK12}  P.G. Casazza, G. Kutyniok (eds.), {\it Finite frames: theory and applications}, Birkhäuser, Boston (2012).

\bibitem{K1} A. Chebira, J. Kova\v{c}evi\'c,  {\it Life beyond bases: The advent of frames
(Part I)}, IEEE Signal Process. Mag. 24 (2007), no. 4,  86-104.


\bibitem{K2}  A. Chebira, J. Kova\v{c}evi\'c, {\it Life beyond bases: The advent of frames
(Part II)}, IEEE Signal Process. Mag. 24 (2007), no. 5,  115-125.

\bibitem{[Chrbook]} O. Christensen, {\it An introduction to frames and Riesz bases},
Birkh\"auser, Boston,  (2003).


\bibitem{CFP03} M.T. Chu, R.E. Funderlic, R.J. Plemmons, {\it Structured low rank approximation}, Linear Algebra Appl. 366 (2003), 157-172.   


\bibitem{EY36} C. Eckart, G. Young,{\it The approximation of one matrix by another of lower rank}, Psychometrika 
1 (1936),  211-218.

\bibitem{F} H. Feichtinger, W. Kozek, F. Luef,  {\it Gabor analysis over finite abelian groups}, Appl. Comput. Harmon. Anal. 26 (2009), no. 2, 230-248.


\bibitem{FPT02} M. Frank, V. Paulsen, T. Tiballi, {\it Symmetric Approximation
of frames and bases in Hilbert Spaces}, Trans. Amer. Math. Soc. 354 (2002), 777-793.

\bibitem{GL} J.A. Goldstein, M. Levy, {\it Linear algebra and quantum chemistry}, Amer. Math. Monthly 98  (1991), no. 8,  710-718.


\bibitem{HM63} P.R. Halmos, J.E. Mc Laughlin, {\it Partial isometries}, Pacific J. Math. 13 (1963), 585-596.


\bibitem{H1} D. Han, {\it Approximations for Gabor and wavelet frames}, Trans. Amer. Math. Soc. 355 (2003),
3329-3342.

\bibitem{H2} D. Han,  {\it Tight frame approximation for multi-frames
and super-frames}, J. Approx. Theory 129 (2004), 78-93.



\bibitem{H08} N.J. Higham,  {\it Functions of matrices. Theory and computation}, Society for Industrial and Applied
Mathematics, Philadelphia, PA  (2008). 


\bibitem{HJ} R.A. Horn, C.R. Johnson, {\it Matrix Analysis}, Cambridge University Press, New York (1985).


\bibitem{HJ91} R.A. Horn, C.R. Johnson, {\it Topics in Matrix Analysis}, Cambridge University Press, New York (1991).


\bibitem{JS02} A. Janssen, T. Strohmer, {\it Characterization and computation of canonical tight windows for
Gabor frames}, J. Fourier Anal. Appl. 8 (2002), 1-28.


\bibitem{FH55} Ky Fan, A.J. Hoffman, {\it Some metric inequalities in the space of matrices}, Proc. 
Amer. Math. Soc. 6 (1955), 111-116. 


\bibitem{LZ06} B. Laszkiewicz, K. Zi\c{e}tak, {\it Approximation of matrices and a family of Gander methods for polar decomposition}, BIT 46 (2006), no. 2, 345-366.


\bibitem{L47} P.-O. L\"owdin, Arkiv Mat. Astr. Fysik 35 A (1947), no. 9, 1-10. 


\bibitem{L70} P.-O. L\"owdin, {\it On the nonorthogonality problem}, Adv. Quantum Chem. 5 (1970), 185-199.




\bibitem{M89} P.J. Maher, {\it Partially isometric approximation of positive operators}, Illinois J. Math. 33
(1989),  227-243.


\bibitem{LM99} C.-K. Li, R. Mathias, {\it The Lidskii-Mirsky-Wielandt theorem--additive and multiplicative versions}, Numer. Math. 81 (1999), 377-413.


\bibitem{MRS14} P.G. Massey, M.A. Ruiz, D. Stojanoff, {\it Optimal frame completions}, Adv. Comput. Math. 40 (2014), 1011-1042.



\bibitem{M60} L. Mirsky, {\it Symmetric gauge functions and unitarily invariant norms}, Quart.
J. Math. 11 (1960), no. 1,  50-59.


\bibitem{R80} C.R. Rao, {\it Matrix approximations and reduction of dimensionality in multivariate  
statistical analysis}. In: V.P.R. Krishnaiah (ed.), Multivariate Analysis, 3-22. North Holland, 
Amsterdam (1980).

\bibitem{RS} A. Ron, Z. Shen, Frames and stable bases for shift invariant subspaces of $L^2(R^d)$, Canad. J. Math. 47 (1995)1051?1094.

\bibitem{W86} P.Y. Wu, {\it Approximation by partial isometries}, Proc. Edinb. Math. Soc. 29 (1986),
 255-261.

\end{thebibliography}
\end{document}